\documentclass[a4paper]{amsart}
\usepackage{graphicx} % Required for inserting images
\usepackage{amsmath,amsfonts,amssymb}
\usepackage[left=3.5cm,right=3.5cm,top=2cm,bottom=2cm]{geometry}
\usepackage[pagebackref]{hyperref}

\DeclareMathOperator{\tr}{Tr}
\newcommand{\1}{\mathbf{1}}
\newcommand{\p}{\mathbb{P}}
\newcommand{\E}{\mathbb{E}}
\newcommand{\N}{\mathbb{N}}
\newcommand{\R}{\mathbb{R}}
\newcommand{\MS}{\mathcal{M}_{sym}^{d\times d}}

\newcommand{\id}{\mathbf{I}}
\newcommand{\Var}{\mathrm{Var}}
\newcommand{\supp}{\mathrm{supp}}
\newtheorem{theorem}{Theorem}[section]
\newtheorem{lemma}[theorem]{Lemma}
\newtheorem{cor}[theorem]{Corollary} % added
\theoremstyle{definition}
\newtheorem{defi}[theorem]{Definition}

\theoremstyle{remark}
\newtheorem{remark}[theorem]{Remark}

\numberwithin{equation}{section}

\title[Matrix concentration for dependent binary random variables]{Matrix concentration inequalities for dependent binary random variables}
\author{Radosław Adamczak
}
\thanks{R.A.'s research was partially supported by The National Science Center, Poland, grant IMPRESS-U 2023/05/Y/ST1/00188}
\author{Ioannis Kavvadias}
\address{University of Warsaw, Poland}
\email{r.adamczak@mimuw.edu.pl, i.kavvadias@uw.edu.pl}
\begin{document}

\begin{abstract} We prove Bernstein-type matrix concentration inequalities for linear combinations with matrix coefficients of binary random variables satisfying certain $\ell_\infty$-independence assumptions, complementing recent results by Kaufman, Kyng and Solda. For random variables with the Stochastic Covering Property or Strong Rayleigh Property we prove estimates for general functions satisfying certain direction aware matrix bounded difference inequalities, generalizing and strengthening earlier estimates by the first-named author and Polaczyk.

We also demonstrate a general decoupling inequality for a class of Banach-space valued quadratic forms in negatively associated random variables and combine it with the matrix Bernstein inequality to generalize results by Tropp, Chr\'etien and Darses, and Ruetz and Schnass, concerning the operator norm of a random submatrix of a deterministic matrix, drawn by uniform sampling without replacements or rejective sampling, to submatrices given by general Strong Rayleigh sampling schemes.

\medskip


\medskip

\end{abstract}

\maketitle

\section{Introduction}

Concentration inequalities for sums of independent random matrices have found numerous applications in statistics, computer science, high-dimensional probability, functional analysis and quantum information theory. Given the great progress in this direction achieved in the last  years, it is difficult to give a full account of the literature, so we will mention here just a few landmark papers, establishing the most important tools of the theory: the work on the non-commutative Khintchine inequalities by Lust-Piquard and Pisier \cite{MR859804,MR1150376}, the paper \cite{MR1694526} by Rudelson, where these inequalities were used to derive  bounds on the accuracy of empirical approximation of covariance matrices, the article by Ahlswede and Winter establishing the first matrix Chernoff bounds \cite{MR1889969}, the subsequent work by Oliveira \cite{MR2653725} and finally the work \cite{MR2946459} by Tropp, strengthening and polishing previous estimates and providing a whole range of inequalities valid under various assumptions on the random matrices involved. We refer the reader to the last reference for a more detailed description of the history of the problem up to 2012. More recent developments involve connections with free probability and refinements of previous inequalities which do not outperform them in general, but allow for elimination of certain logarithmic factors in situations where one may expect that the problem is indeed well approximated by its free counterpart \cite{MR3768857,MR4635836,bandeira2024matrixconcentrationinequalitiesfree,MR4823211}.

When one drops the assumption of independence, the theory is so far less developed, however there are also several results, covering various type of assumptions concerning the dependence structure. One of the first results is the matrix Freedman inequality for martingales due to Tropp \cite{MR2802042}. Another lines of research establish bounds for matrices corresponding to mixing sequences, Markov chains or more specifically random walks on graphs \cite{MR3493552,MR3826320,MR4838296,vanwerde2023matrixconcentrationinequalitiesdependent}, concentration under matrix counterparts of classical functional inequalities \cite{MR4108222,MR4260507,MR4216521} and inequalities for matrix-valued functions of random vectors with values in the discrete cube $\{0,1\}^n$, satisfying some Dobrushin-type or negative dependence conditions \cite{MR4108222,MR3899605,MR4683375,kaufman2022scalar}.

This article belongs to the last group. We are interested in concentration inequalities for matrix valued functions on the discrete cube under non-product measures. Motivations for investigating concentration for such models come from combinatorics, physics, computer science and statistics, as dependent binary random variables arise in the analysis of random combinatorial structures (for instant random bases of matroids, in particular random spanning trees of the graph, \cite{DBLP:conf/stoc/FederM92,MR3899605}), discrete models of statistical physics (see, e.g., the monographs \cite{MR3752129,MR2108619}), randomized algorithms (see e.g., \cite{MR2547432}), and survey sampling, where they allow for selecting a sample from a population with prescribed inclusion probabilities (see, e.g., \cite{MR2225036,MR3000852,MR4010964,MR3619696}).

In subsequent sections we develop several inequalities working under different (but interrelated) types of assumptions. Since precise formulation of our theorems would require additional definitions, we postpone the statements to Section \ref{sec:results}, and here we just describe the nature of the estimates and its relations with previous results.

The inequalities we present are of the form
\begin{displaymath}
  \p(\|Z\| \ge t)\le 2d\exp\Big(-\min\Big(\frac{t^2}{a^2},\frac{t}{b}\Big)\Big),
\end{displaymath}
where $Z$ is a matrix-valued function on $\{0,1\}^n$, $\p(\cdot)$ is a probability on $\{0,1\}^n$, satisfying appropriate weak or negative dependence assumptions, and $a,b$ are parameters responsible for the Gaussian and exponential rate of tail decay, which hold respectively for \emph{small} and \emph{large} values of $t$. The right-hand side resembles the tail decay given by the classical Bernstein inequality for sums of independent random variables, therefore we call such class of estimates \emph{Bernstein-type inequalities}. The parameter $b$ is usually related to some $\ell_\infty$ bounds on the increments of the function $f$, while $a$ should be thought of as a variance proxy and be of a weaker, informally speaking, integrated form.

Our approach is inspired by the recent article \cite{kaufman2022scalar} by Kaufman, Kyng and Solda who were able to adapt to the dependent setting classical tools going back to \cite{MR1889969} and prove bounds for random variables of the form $\sum_{i=1}^n X_i A_i$ where $X_i$'s are binary random variables satisfying an appropriate Dobrushin-type condition (see Section \ref{sec:assumptions} below for all the definitions and assumptions), and $A_i$'s are nonnegative definite matrices. By refining their approach we obtain Bernstein-type inequalities for general (not neccessarily nonnegative definite) matrices which involve as a subgaussian parameter the quantity $\|\E \sum_{i=1}^n X_i A_i^2\|$, just as in the corresponding estimates in the independent setting (Theorems \ref{thm:Bernstein}, Corollary \ref{cor:Bernstein-non-hom}). Under stronger assumptions of Stochastic Covering Property or Strong Rayleigh Property, by combining the method of \cite{kaufman2022scalar} with the approach from the paper by the first named author and Polaczyk \cite{MR4683375}, we obtain a Bernstein type inequality for general matrix-valued functions satisfying a McDiarmid type increment conditions, answering a question which was left open in \cite{MR4683375} (Theorem \ref{thm:general-function-SCP}, Corollary \ref{cor:general-function-SRP}). Finally, we address the problem of random subsampling of deterministic matrices, and obtain norm estimates for random submatrices selected via random vectors with Strong Rayleigh Property (Theorem \ref{thm:matrix-sampling}). It is a generalization of the results by Tropp \cite{MR2379999}, Chr\'etien and Darses \cite{MR2929804} for uniform sampling without replacement, and more recent inequalities due to Ruetz and Schnass \cite{MR4296760} for rejective sampling, to a wider family of efficient sampling schemes. To this end we also develop a decoupling inequality for a class of Banach space-valued quadratic forms in negatively associated binary variables, which is of independent interest (Theorem \ref{thm:decoupling}).

The organization of the article is as follows. In Section \ref{sec:notation} we introduce the basic notation and discuss the notions of weak and negative dependence used in our assumptions. In Section \ref{sec:results} we formulate our main theorems and provide a comparison with earlier results. More specifically, Section \ref{sec:Bernstein} is devoted to Bernstein-type inequalities for $\ell_\infty$-independent random vectors, Section \ref{sec:SCP} to inequalities for more general functions of vectors with Stochastic Covering Property, and Section \ref{sec:decoupling} to decoupling inequalities and applications to the problem of sampling random submatrices. Section \ref{sec:proofs} contains proofs of all our theorems.

\section{Notation and preliminaries}\label{sec:notation}

\subsection{Basic notation}
We will use the notation $[n] = \{1,\ldots,n\}$. The cardinality of a set will be denoted by $|\cdot|$. For $x \in \{0,1\}^n$ by $\supp(x)$ we will denote the support of $x$, i.e., the set $\{i\in[n]\colon x_i =1\}$. For a set $I \subseteq [n]$, we write $I^c$ for $[n]\setminus I$.

By $\MS$ we will denote the space of symmetric $d\times d$ matrices with real coefficients. The identity matrix will be denoted by $\id$ (the size of the matrix will be always clear from the context).

By $e_1,\ldots,e_n$ we will denote the standard basis in $\R^n$. The standard Euclidean norm in $\R^n$  will be denoted by $|\cdot|$ (which should not lead to confusion with cardinality of a set) and the Euclidean unit sphere in $\R^n$ by $S^{n-1}$.
For $p \ge 1$, by $\ell_p^n$ we will denote the Banach space $(\R^n, \|\cdot\|_p)$ where $\|x\|_p = (|x_1|^p+\cdots+|x_n|^p)^{1/p}$. For an $m\times n$ matrix $A$ by $\|A\|_{\ell_p\to \ell_q}$ we will denote the operator norm of $A$ seen as an operator between $\ell_p^n$ and $\ell_q^n$, i.e., $\|A\|_{\ell_p\to \ell_q} = \sup_{\|x\|_p \le 1} \|Ax \|_{q}$. When $p,q = 2$, we will drop the subscript $\ell_2\to \ell_2$ and write simply $\|\cdot\|$. Thus $\|A\| = \|A\|_{\ell_2 \to \ell_2} = \sup_{x,y \in S^{n-1}} |x^T A y|$. We will write $\preceq$ for the positive-semidefinite order on symmetric matrices. For $x,y \in \{0,1\}^n$ by $x\oplus y$ we will denote the coordinate-wise sum of $x$ and $y$ modulo 2.

We say that a random vector $X = (X_1,\ldots,X_n)$ with values in $\{0,1\}^n$ is $k$-homogeneous if with probability one, $\sum_{i=1}^n X_i = k$. For a random vector $Y$, by $\mu_Y$ we will denote its law. If $\mu_Y = \mu$ we will also write $Y \sim \mu$. Similarly, if $X,Y$ have the same law, we will write $X \sim Y$.
For $I \subseteq [n]$ and a $\{0,1\}^n$-valued random vector $X = (X_1,\ldots,X_n)$ we set $X_I = (X_i)_{i \in I}$. For a random variable $X$, by $\sigma(X)$ we will denote the $\sigma$-field generated by $X$.

\subsection{Assumptions on random vectors}\label{sec:assumptions}

Let us now describe various notions of dependence between binary random variables, which we will use in the article. We will start with $\ell_\infty$-independence, which relies on bounding an appropriate operator norm of an interdependence matrix. Various versions of this notion has appeared in the literature starting from the seminal work by Dobrushin (see, e.g., \cite{MR1756011,MR3248197,MR4091094} for applications to concentration of measure phenomenon). In the context of matrix concentration inequalities the version we present below has been recently introduced by Kaufman, Kyng and Solda \cite{kaufman2022scalar}.

Let $\mu$ be a probability measure on $\{0,1\}^n$. Consider a random vector $X = (X_1,\ldots,X_n)\sim \mu$. For $\Lambda \subseteq [n]$, and $\sigma\in \{0,1\}^\Lambda$, consider a matrix $\Psi_\mu^{\Lambda,\sigma} = (\Psi_\mu^{\Lambda,\sigma}(i,j))_{i,j\in [n]}$ defined as
\begin{displaymath}
  \Psi_\mu^{\Lambda,\sigma}(i,j) = \p(X_j=1|X_i = 1, \forall_{\ell \in \Lambda} X_\ell = \sigma_\ell) -  \p(X_j=1|X_i = 0, \forall_{\ell \in \Lambda} X_\ell = \sigma_\ell)
\end{displaymath}
whenever $\p(X_i = 1, \forall_{\ell \in \Lambda} X_\ell = \sigma_\ell), \p(X_i = 0, \forall_{\ell \in \Lambda} X_\ell = \sigma_\ell) > 0$ and
$ \Psi_\mu^{\Lambda,\sigma}(i,j) = 0$ otherwise.

Following \cite{kaufman2022scalar} we will also consider a matrix $\mathcal{I}_\mu^\Lambda = (\mathcal{I}_\mu^\Lambda(i,j))_{i,j\in [n]}$, given by
\begin{displaymath}
  \mathcal{I}_\mu^\Lambda(i,j) = \p(X_j = 1|X_i = 1, \forall_{\ell \in \Lambda} X_\ell = 1) - \p(X_j = 1|\forall_{\ell \in \Lambda} X_\ell = 1),
\end{displaymath}
whenever $\p(X_i = 1, \forall_{\ell \in \Lambda} X_\ell = 1) > 0$, and $\mathcal{I}_\mu^\Lambda(i,j) = 0$ otherwise.

For a matrix $A = (a_{ij})_{i,j\in [n]}$, by $\|A\|_{\ell_\infty \to \ell_\infty}$ we will denote the norm of $A$ seen as an operator from the space $\ell_\infty^n$ to itself, i.e.,
\begin{displaymath}
  \|A\|_{\ell_\infty \to \ell_\infty} = \max_{1\le i \le n} \sum_{j=1}^n |a_{ij}|.
\end{displaymath}
\begin{defi}[$\ell_\infty$-independence] We will say that a probability measure $\mu$ on $\{0,1\}^n$ (equiv. a random vector $X$ with law $\mu$) is

\begin{itemize}
\item (one-sided) $\ell_\infty$-independent with parameter $D$ if
for every $\Lambda \subseteq [n]$,
\begin{displaymath}
  \|\mathcal{I}_\mu^{\Lambda}\|_{\ell_\infty \to \ell_\infty} \le D,
\end{displaymath}
\end{itemize}
and
\begin{itemize}
\item two-sided $\ell_\infty$-independent with parameter $D$ if
for every $\Lambda \subseteq [n]$ and $\sigma \in \{0,1\}^\Lambda$,
\begin{displaymath}
  \|\Psi_\mu^{\Lambda,\sigma}\|_{\ell_\infty \to \ell_\infty} \le D.
\end{displaymath}
\end{itemize}
\end{defi}

\begin{remark}\label{re:D-at-least-1} By considering $i = j$ we can see that if a random vector $X$ is two-sided $\ell_\infty$-independent and $X$ is not deterministic, then $D \ge 1$. It is not difficult to see that this is not true for one-sided $\ell_\infty$-independence, for instance a random vector with i.i.d. coordinates, taking values 1 with probability very close to 1 is $\ell_\infty$-independent with parameter $D \simeq 0$. If we however restrict our attention to $k$-homogeneous random vectors, then again, for non-deterministic vectors, we have $D \ge 1$ (see Lemma \ref{le:D-lower-bound} below).

Moreover the two-sided $\ell_\infty$-independence implies its one-sided version, but the reverse implication in generality does not hold (see \cite[Appendix C]{kaufman2022scalar}).
\end{remark}

In \cite{kaufman2022scalar} the authors provide many examples of random vectors satisfying the above definitions, in particular examples coming from the theory of negative dependence (see below) and Gibbs measures coming from statistical physics. We refer the reader to this article for the proofs and here we just mention that the examples include the monomer-dimer model, the Ising/Potts model, list colouring, stable distributions and distributions with the Stochastic Covering Property. The last example is the most important for us and we recall it in
Definition \ref{defi:SCP} below. Let us also remark that the conditions from the definitions of $\ell_\infty$-independence for $\Lambda = \emptyset$ are related to the notion of spectral independence \cite{MR4839668}. In particular, examples of measures satisfying $\ell_\infty$-conditions may be often obtained from examples satisfying spectral independence. For instance, following conditionally the arguments of \cite[Chapter 2, Example 5]{KuiKuiLiuPhd}, one can prove that if $X$ is homogeneous and for every $\Lambda \subseteq [n]$, and all $i,j \in \Lambda^c$, conditionally on the event $\{\forall_{\ell\in \Lambda} X_\ell = 1\}$, the variables $X_i,X_j$ are negatively correlated, then $X$ is one-sided $\ell_\infty$-independent with $D = 2$.

\vskip0.7cm

The next definitions are related to the theory of negative dependent random variables. The first one, we would like to introduce, is negative association. It was introduced in a more general context of real-valued variables in \cite{MR684886} and since then it has become one of the basic notions of the theory of negative dependence.

\begin{defi}[Negative association] Let $X$ be a random vector with values in $\R^n$. We will say that $X$ is negatively associated if for any $\emptyset \neq I \subsetneq [n]$ and (coordinate-wise) non-decreasing functions $f \colon \{0,1\}^I \to \R$ and $g\colon \{0,1\}^{I^c} \to \R$, we have
\begin{displaymath}
  \E f(X_I)g(X_{I^c}) \le \E f(X_I) \E g(X_{I^c}).
\end{displaymath}
\end{defi}

In some of our results we will work under an assumption stronger than negative association, namely the Strong Rayleigh Property.

\begin{defi}[Rayleigh Properties]
An $\{0,1\}^n$-valued random vector $X = (X_1,\ldots,X_n)$ has the Rayleigh property if its generating polynomial
\begin{displaymath}
  F_X(x_1,\ldots,x_n) = \E x_1^{X_1}\cdots x_n^{X_n}
\end{displaymath}
satisfies
\begin{displaymath}
  F(x)\frac{\partial^2}{\partial x_i\partial x_j}F(x) \le \frac{\partial F}{\partial x_i}(x)\frac{\partial F}{\partial x_j}(x)
\end{displaymath}
for all $i,j \in [n]$ and $x = (x_1,\ldots,x_n) \in [0,\infty)^n$.

If the above inequality holds for all $x \in \R^n$, then we say that $X$ has the Strong Rayleigh Property.
\end{defi}

The above definition was introduced by Borcea, Br\"and\'en and Liggett in \cite{MR2476782}. They proved that many distributions of interest satisfy the SRP. It is known that SRP implies negative association. The by now classical examples include  uniform measures on bases of balanced matroids (in particular random uniform spanning trees of a graph), determinantal measures, Bernoulli random variables conditioned on their sum, measures related to exclusion dynamics, determinantal measures.

Finally, let us recall the definition of Stochastic Covering Property (SCP). It was introduced by Pemantle and Peres \cite{MR3197973}, who showed that it is implied by the SRP, but there is no implication in the other direction. Recall that $e_1,\ldots, e_n$ is the standard basis in $\R^n$.

\begin{defi}[Stochastic covering property]\label{defi:SCP}
Let $x,y\in\{0,1\}$. We say that $x$ covers $y$, which we denote by $x \triangleright y$, if $x=y$ or $x=y+e_i$ for some $i\in [n]$.

Let $\mu$ be a probability measure on $\{0,1\}^n$ and let $X \sim \mu$. We will say that $\mu$ (equivalently $X$) satisfies the Stochastic Covering Property (abbrev. SCP) if for any $\Lambda \subset [n]$ and any $x,y\in\{0,1\}^n$ such that $\p(X_\Lambda=x_\Lambda),\p(X_\Lambda=y_\Lambda)>0$ and $x_\Lambda \triangleright y_\Lambda$, there exists a coupling $(Y,Q)$ between the conditional distributions $\p(X_{\Lambda^c} \in \cdot \,\vert\, X_\Lambda = y_\Lambda)$ and $\p(X_{\Lambda^c} \in \cdot \,\vert\, X_\Lambda = x_\Lambda)$ such that with probability one, $Y \triangleright Q$.
\end{defi}

Kaufman, Kyng and Solda proved that for $k$-homogeneous measures the SCP implies two-sided $\ell_\infty$-independence with parameter $D = 2$. 

\section{Main results}\label{sec:results}

\subsection{Matrix Bernstein inequality for $\ell_\infty$-independent random vectors}\label{sec:Bernstein}
In this section we will present results for linear combinations of binary random variables with matrix coefficients, corresponding to non-commutative Bernstein-type inequalities, obtained in the independent setting by Tropp \cite{MR2946459} (see also \cite{MR1889969,MR2653725} for important earlier contributions). We will work under the assumption of $\ell_\infty$-independence. In the special case of nonnegative definite matrices, the inequalities we obtain allow to recover the results by Kaufman, Kyng and Solda \cite{kaufman2022scalar}. We should, however, stress that our approach builds extensively on their ideas.

\begin{theorem}\label{thm:Bernstein}
Let $X = (X_1,\ldots,X_n)$ be an $\{0,1\}^n$-valued random vector and consider $A_1,\ldots,A_n \in \MS$. Define the random variable
\begin{displaymath}
  Z = \sum_{i=1}^n X_i A_i.
\end{displaymath}
Assume that $X$ is $k$-homogeneous and one-sided $\ell_\infty$-independent with parameter $D$. Then, for some absolute constant $C$ and all $t > 0$,

\begin{align}\label{eq:Bernstein}
  \p(\|Z - \E Z\| \ge t) \le 2d\exp\Big(-\frac{1}{CD^2}\min\Big(\frac{t^2}{\sigma^2},\frac{t}{R}\Big)\Big),
\end{align}
where $\sigma^2 = \|\E \sum_{i=1}^n X_i A_i^2\|$ and $R = \max_{i\le n} \|A_i\|$.
\end{theorem}

\begin{remark} The quantity $\E \sum_{i=1}^n X_i A_i^2 = \sum_{i=1}^n p_iA_i^2$, where $p_i = \p(X_i=1)$. It can be thought of as a matrix variance proxy for $Z$. In fact since $\sum_{i=1}^n X_i = k = \E \sum_{i=1}^n X_i$, it can be replaced by $\E \sum_{i=1}^n X_i (A_i - A)^2$ for any deterministic matrix $A$ at the cost of replacing the constant $R$ by $R + \|A\|$. Choosing $A = \frac{1}{k}\sum_{i=1}^n p_i A_i$ we obtain
\begin{displaymath}
\E \sum_{i=1}^n X_i (A_i - A)^2 = k\Big(\frac{1}{k}\sum_{i=1}^n p_i A_i - \Big(\frac{1}{k}\sum_{i=1}^n p_iA_i\Big)^2\Big) \preceq \E \sum_{i=1}^n X_i A_i^2.
\end{displaymath}
Observe that in this case $\|A\| \le R$. One can also note that the left-hand side above equals to $k\Var(A_I)$, where $I$ is a random element of $[n]$, which conditionally on $X$ is drawn uniformly from $\supp(X) = \{i\colon X_i = 1\}$. Its operator norm can be significantly smaller than that of $\E \sum_{i=1}^n X_i A_i$, in particular it vanishes when all $A_i$'s are equal (in which case $Z$ is deterministic).
\end{remark}

\begin{remark} As seen from the proofs, presented in Section \ref{sec:proofs-Bernstein}, in the above theorem one can take $C = 35$. Also in subsequent inequalities one can obtain explicit constants from our proofs. Since they are most likely rather far from optimal, we do not specify them.
\end{remark}

\begin{remark}
We have stated our results for symmetric matrices, but the standard hermitization trick (known also as dilation) allows to obtain results for non-symmetric square matrices or rectangular matrices. Since this method is by now standard, we do not state explicitly the results, and refer to \cite[Section 2.6]{MR2946459} for details.
\end{remark}

Similarly as in \cite{kaufman2022scalar}, under the stronger assumption that $X$ is two-sided $\ell_\infty$-independent, we can drop the assumption of homogeneity.
\begin{cor}\label{cor:Bernstein-non-hom}
Let $X = (X_1,\ldots,X_n)$ be an $\{0,1\}^n$-valued random vector and consider $A_1,\ldots,A_n \in \MS$. Define the random variable
\begin{displaymath}
  Z = \sum_{i=1}^n X_i A_i.
\end{displaymath}
Assume that $X$ is two-sided $\ell_\infty$-independent with parameter $D$. Then, for some absolute constant $C$ and all $t > 0$,
\begin{displaymath}
  \p(\|Z - \E Z\| \ge t) \le 2d\exp\Big(-\frac{1}{CD^2}\min\Big(\frac{t^2}{\sigma^2},\frac{t}{R}\Big)\Big),
\end{displaymath}
where $\sigma^2 = \|\E \sum_{i=1}^n X_i A_i^2\|$ and $R = \max_{i\le n} \|A_i\|$.
\end{cor}

In the special case, when the matrices $A_1,\ldots,A_n$ are nonnegative definite, the following, simplified bound may be sometimes sufficient. We will use it in particular to obtain bounds concerning sampling submatrices of a deterministic matrix (see Section \ref{sec:decoupling}).

\begin{cor}\label{cor:KKS-bound}
Under the assumptions of Theorem \ref{thm:Bernstein} or Corollary \ref{cor:Bernstein-non-hom} on the random vector $X$, for any nonnegative definite $A_1,\ldots,A_n$, the random variable $Z = \sum_{i=1}^n X_iA_i$ satisfies for all $t > 0$,
\begin{displaymath}
  \p(\|Z\| \ge \|\E Z\| + t) \le 2d\exp\Big(-\frac{1}{CD^2}\min\Big(\frac{t^2}{\|\E Z\|R},\frac{t}{R}\Big)\Big),
\end{displaymath}
where $C$ is an absolute constant and $R = \max_{i\le n}\|A_i\|$.
\end{cor}

\begin{remark} Let us end this section with a brief comparison of our estimates with the inequalities by Kaufman, Kyng and Song \cite{kaufman2022scalar}, which were a direct source of inspiration for us.
They considered nonnegative definite matrices $A_i$, with $\|A_i\|\le 1$ and proved that for $\delta \in [0,1]$
\begin{align*}
  \p(\|Z\| \ge (1+\delta) \|\E Z\|) & \le d \exp\Big(-\frac{\delta^2 \|\E Z\|}{CD^2}\Big)\\
  \p(\mu_{\min}(Z) \le (1-\delta)\mu_{\min}(\E Z)) & \le d\exp\Big(-\frac{\delta^2 \mu_{\min}(\E Z)}{CD^2}\Big),
\end{align*}
where $\mu_{\min}(A)$ is the smallest eigenvalue of the matrix $A \in \MS$.

It is easy to see that their upper bound follows from Corollary \ref{cor:KKS-bound} (up to the value of the constant $C$). One should stress that this corollary could be also obtained from the proof in \cite{kaufman2022scalar}. On the other hand, the general inequality \eqref{eq:Bernstein} of Theorem \ref{thm:Bernstein} and Corollary \ref{cor:Bernstein-non-hom} requires a modification of the approach from \cite{kaufman2022scalar} and leads in general to stronger estimates. While the strength of both inequalities in the positive definite case considered in the CS applications discussed in \cite{kaufman2022scalar} is comparable, in general the parameter $\sigma^2$ involving matrices $A_i^2$ may be much smaller than $\|\E Z\|$, leading to stronger upper bounds.

As for the lower bound on the smallest eigenvalue value in the nonnegative definite case, inequality \eqref{eq:Bernstein} would yield
\begin{multline*}
  \p(\mu_{\min}(Z) \le (1-\delta)\mu_{\min}(\E Z)) \\
  \le 2d \exp\Big(-\frac{1}{CD^2}\min\Big(\frac{\delta^2 \mu_{min}(\E Z)^2}{\sigma^2},\delta\mu_{\min}(\E Z)\Big)\Big),
\end{multline*}
which is not directly comparable with the lower bound of Kaufman, Kyng and Solda. One can easily construct examples in which $\mu_{\min}(\E Z)$ is much larger than $\sigma^2$ (in which case our bound is better) and in which $\mu_{\min}(\E Z)$ is much smaller than $\sigma^2$ (then, the bound by Kaufman, Kyng and Solda outperforms ours).

Let us mention that the main inequality on the Laplace transform used in our proofs, i.e., formula \eqref{eq:trace-induction} is (up to the value of constants) a strengthening of the corresponding estimate in \cite[Lemma 5.2]{kaufman2022scalar}, so it also allows to obtain the lower bound from \cite{kaufman2022scalar}.
\end{remark}

\subsection{Concentration for Lipschitz functions of random vectors withthe  SCP}\label{sec:SCP}

In this section we will present concentration results for more general matrix-valued functions under an additional assumption that the random vector $X$ has the SCP. Recall the notation $\oplus$ for the coordinate-wise mod 2 addition of binary vectors and that $e_1,\ldots,e_n$ is the standard basis in $\R^n$.

\begin{theorem}\label{thm:general-function-SCP} Assume that $X$ is a $k$-homogeneous $\{0,1\}^n$-valued random vector, satisfying the SCP. Let $f \colon \{0,1\}^n \to \MS$, such that for some nonnegative definite matrices $A_i \in \MS$, and all $i \in [n]$,
\begin{align}\label{eq:increment-assumption}
  (f(x) - f(x\oplus e_i))^2 \preceq A_i^2.
\end{align}
Then for every $t \ge 0$,
\begin{displaymath}
  \p(\|f(X) - \E f(X)\| \ge t) \le 2d\exp\Big(-\frac{1}{C}\min\Big(\frac{t^2}{\sigma^2},\frac{t}{R}\Big)\Big),
\end{displaymath}
where $\sigma^2 = \|\E \sum_{i=1}^n X_i A_i^2\|$ and $R = \max_{i\le n} \|A_i\|$.
\end{theorem}

\begin{remark}
The above theorem improves \cite[Theorem 2.8]{MR4683375}. First, it allows to replace the Strong Rayleigh Property assumed there, by the weaker notion of Stochastic Covering Property. Second, it eliminates spurious $\log(ek)$ factors in the estimate.
\end{remark}

\begin{remark} Clearly, Theorem \ref{thm:general-function-SCP} covers in particular the class of the form $f(x) = \sum_{i=1}^n x_iC_i$, considered in the previous section (with $A_i = \sqrt{C_i^2}$).
Another example of a function to which one can apply the above theorem is
\begin{displaymath}
  f(x) = \sum_{i=1}^{n-1} C_i x_ix_{i+1}.
\end{displaymath}
In this case, it is not difficult to see (e.g., using Lemma \ref{le:op-conv} below), that one can take $A_1 = \sqrt{C_i^2}$, $A_i = \sqrt{2(C_i^2 + C_{i+1}^2)}$ for $i = 2,\ldots,n-1$ and $A_{n} = \sqrt{C_n^2}$.
\end{remark}

In fact, for measures with the SRP, we can obtain the conclusion of the above theorem without assuming homogeneity. Indeed, in \cite[Theorem 4.2]{MR2476782} it is shown that if an $\{0,1\}^n$-valued random vector $X$ has the SRP, then its law is a projected homogeneous Rayleigh measure, i.e., $X$ can be embedded as the first $n$ coordinates of some $\{0,1\}^m$-valued homogeneous random vector with the Rayleigh Property. Moreover, by \cite[Proposition 2.1]{MR3197973}, every projected homogeneous Rayleigh measure (so in particular every homogeneous Rayleigh measure) satisfies the SCP. Therefore, the above theorem yields the following corollary.

\begin{cor}\label{cor:general-function-SRP}
The conclusion of Theorem \ref{thm:general-function-SCP} holds if $X$ is a (not necessarily homogeneous) $\{0,1\}^n$-valued random vector satisfying the SRP.
\end{cor}

\begin{remark} The above corollary applied in dimension one generalizes (up to constants) \cite[Theorem 3.2]{MR3197973} by Pemantle and Peres, which states that for $X$ with the SRP and a real-valued function $f\colon \{0,1\}^n \to \R$, 1-Lipschitz with respect to the Hamming distance, and $t > 0$,
\begin{displaymath}
  \p(|f(X) - \E f(X)| \ge t) \le 5\exp\Big(-\frac{t^2}{16(2\E N + t)}\Big),
\end{displaymath}
where $N = \sum_{i=1}^n X_i$.

\vskip0.7cm

In \cite{kaufman2022scalar} Kaufman, Kyng and Solda ask whether under some weak dependence assumptions it is possible to obtain an inequality of the type
\begin{align}\label{eq:KKF-question}
  |f(X) - \E f(X)| \le C\sqrt{\E f(X) \log(1/\delta)}
\end{align}
with probability at least $1-\delta$ for all nonnegative functions which are 1-Lipschitz with respect to the Hamming distance.

We would like to argue here that in general such an inequality cannot hold for all $\delta \in (0,1)$ with universal $C$. This is witnessed by the independent case, when $X_i$ are i.i.d. with $\p(X_i=1) = 1/n$. Indeed, in this case for $f(x) = x_1+\ldots+x_n$, the random variable $f(X)$ converges weakly to the Poisson distribution with parameter 1, which does not have a subgaussian tail. More precisely, $\E f(X) = 1$, and for $t \in \N$, some constant $K > 0$ arbitrary $c > 0$ and $n,t$ large, we get
\begin{displaymath}
  \p(|f(X) - \E f(X)| \ge t) \le \p(f(X) \ge 1 + t) \ge \frac{1}{2(1+ t)!} e^{-1} \ge e^{-K t\log t} \gg e^{-ct^2},
\end{displaymath}
so taking $t = \sqrt{C\log(1/\delta)}$ we can see that for small $\delta$, \eqref{eq:KKF-question} does not hold. Still, in this example, the subgaussian bound holds for large values of $\delta$ (small values of $t$). Theorem \ref{thm:general-function-SCP} provides a subgaussian bound for general functions and small $t$, but with the expectation of $f(X)$ replaced by $\sigma^2\le \E N$.
\end{remark}

\subsection{A decoupling inequality with application to random submatrices}\label{sec:decoupling}

In this section we first discuss a general decoupling inequality for a class of quadratic forms in negatively associated binary variables, which goes beyond the matrix setting and is formulated in terms of arbitrary Banach spaces. Next, we go back to random matrices and consider the problem of bounding from above the norm of a random submatrix of a deterministic matrix selected with a sampling scheme with the Strong Rayleigh Property.

\begin{theorem}\label{thm:decoupling}
Assume that $X = (X_1,\ldots,X_n)$ is a $k$-homogeneous, $\{0,1\}^n$-valued random vector with negatively associated coordinates. Let $(E,\|\cdot\|)$ be a Banach space and assume that $(c_{ij})_{1\le i,j \le n}$ is a matrix with coefficients in $E$, such that $c_{ii} = 0$ for all $i \in [n]$.   Moreover, assume that the function $f\colon \{0,1\}^n\times \{0,1\}^n \to \R$, given by
\begin{displaymath}
    f(x,y) = \Big\| \sum_{i,j=1}^n c_{ij} x_iy_j\Big\|
\end{displaymath}
is coordinate-wise non-decreasing. There exists a universal constant $C$, such that for all $ t>0$,
\begin{displaymath}
    \p\Big(\Big\|\sum_{i,j=1}^n c_{ij}X_iX_j\Big\| \ge t\Big)
    \le C\p\Big(\Big\|\sum_{i,j=1}^n c_{ij}X_iY_j\Big\| \ge t/C\Big),
\end{displaymath}
where $Y$ is an independent copy of $X$.
\end{theorem}

The examples of functions $f$ for which the monotonicity property holds are for instance operator norms of the form
\begin{displaymath}
\Big\|\sum_{1 \le i\neq j \le n} a_{ij} e_ie_j^T x_i y_j\Big\|_{E_1 \to E_2},
\end{displaymath}
where $E_i = (\R^n,\|\cdot\|_i)$, $i=1,2$, are Banach spaces for which the standard basis $(e_i)_{i\in [n]}$ is unconditional (for instance the classical $\ell_p^n$-spaces). Other examples include quadratic forms in nonnegative definite random matrices $c_{ij}$ with $\|\cdot\| = \|\cdot\|_{\ell_2\to \ell_2}$ or Rademacher averages, i.e., functions of the form
\begin{displaymath}
  f(x,y) = \E_{\varepsilon,\varepsilon'} \Big\|\sum_{i,j=1}^n c_{ij}x_iy_j\varepsilon_i \varepsilon_{j}'\Big\|,
\end{displaymath}
where $\varepsilon,\varepsilon'$ are independent Rademacher sequences (the monotonicity for binary $x,y$ follows from Jensen's inequality, in fact the monotonicity in positive $x,y$ also holds by the contraction principle for Rademacher averages, see, e.g., \cite[Theorem 4.4]{MR1102015}). Here $c_{ij}$'s are coefficients from an arbitrary Banach space.

\vskip0.7cm

As we have already mentioned, by \cite[Theorem 4.2]{MR2476782}, every measure with the SRP is a projected homogeneous Rayleigh measure. On the other hand, by \cite[Theorem 4.10]{MR2476782}, projected homogeneous Rayleigh measures are negatively associated. Thus, the above theorem gives the following corollary.

\begin{cor}\label{cor:decoupling-Rayleigh}
The conclusion of Theorem \ref{thm:decoupling} holds if $X$ is a (not necessarily homogeneous) $\{0,1\}^n$-valued random vector satisfying the SRP.
\end{cor}

We remark that we do not know if one can drop the homogeneity assumption in the case of arbitrary negatively associated random variables.

\vskip0.7cm

As an application of the above corollary and Bernstein type inequalities for $\ell_\infty$-dependent measures from Section \ref{sec:Bernstein}, we obtain the following theorem concerning the operator norm of a matrix restricted to a random set of rows and columns, given by the support of a random vector with the SRP. Before we formulate it, let us recall the notation $\|A\|_{\ell_p\to \ell_q}$ for the operator norm of an $m\times n$ matrix seen as an operator from $\ell_p^n$ to $\ell_q^m$. Observe that $\|A\|_{\ell_1\to \ell_2}$ is the maximum Euclidean length of a column of $A$, $\|A\|_{\ell_2\to \ell_\infty}$ the maximum Euclidean length of a row, while $\|A\|_{\ell_1\to \ell_\infty}$ is just the maximum (in absolute value) entry of $A$.

\begin{theorem}\label{thm:matrix-sampling}
Let $X = (X_1,\ldots,X_d)$ be a $\{0,1\}^d$-valued random vector satisfying the SRP.  Cosider a deterministic matrix $H = (h_{ij})_{i,j=1}^d$ with zero diagonal and a random diagonal matrix $\Lambda_X  = Diag(X_1,\ldots,X_d)$. Let $H_X = \Lambda_X H \Lambda_X$ and let $P = Diag(\p(X_1=1),\ldots,\p(X_d=1))$.
There exist universal constants $C,c > 0$, such that for $t > C\|\sqrt{P}H\sqrt{P}\|$,
\begin{align}\label{eq:submatrix}
  \p(\|H_X\| \ge t) \le 2d\exp\Big(-c\min\Big(\frac{t^2}{\|\sqrt{P}H\|_{\ell_1\to \ell_2}^2},\frac{t^2}{\|H\sqrt{P}\|_{\ell_2\to \ell_\infty}^2},\frac{t}{\|H\|_{\ell_1\to\ell_\infty}}\Big)\Big).
\end{align}
\end{theorem}

Results concerning random restrictions of deterministic matrices for a subset of rows and columns of fixed cardinality selected by the uniform sampling without replacement or a subset obtained by i.i.d. binary variables, appeared in the literature in the context of local theory of Banach spaces and the Kadison--Singer problem (see the famous papers \cite{MR890420,MR1124564} by Bourgain and Tzafriri as well as the article \cite{MR2379999} by Tropp, which simplifies some of the proofs), signal processing and statistics \cite{MR2419702,MR2543688,MR2929804,MR4296760}. Theorems resembling Theorem \ref{thm:matrix-sampling} in the case of uniform sampling without replacement appeared in the work by Tropp \cite{MR2379999,MR2419702} (in the language of moments) and in \cite{MR2929804} by Chr\'etien and Darses.

In the recent article \cite{MR4296760}, Ruetz and Schnass consider special vectors with the SRP property, i.e., vectors obtained by conditioning independent Bernoulli variables to have a fixed sum $k$. More precisely, if $Y_1,\ldots,Y_n$ are independent random variables with $\p(Y_i = 1) = 1 - \p(Y_i = 0) = \pi_i$, where $\sum_{i=1}^n \pi_i = k$, then the vector they consider has the law
\begin{displaymath}
  \p(X \in \cdot) = \p\Big(Y \in \cdot \Big|\sum_{i=1}^n Y_i = k\Big).
\end{displaymath}
It is well known that $X$ satisfies the SRP. For this class of random variables Ruetz and Schnass obtain the estimate \eqref{eq:submatrix} (with explicit constants $C,c$) and with the matrix $P$ replaced by $\Pi = Diag(\pi_1,\ldots,\pi_n)$. Such a sampling scheme is known in the literature as rejective sampling or conditional Poisson sampling (see, e.g., \cite{MR2225036,MR4010964}).

The motivation for considering such random restrictions comes from signal processing and the analysis of statistical models in which the signal is represented as a linear combination of randomly chosen atoms from some dictionary. The authors of \cite{MR4296760} argue that when modelling data appearing in applications, choosing a model with non-equal probabilities of inclusion of particular atoms may give better results than the classical approach when the atoms are chosen via uniform sampling with replacement. We refer to their work for examples of applications of their estimates for random restrictions to various sparse approximation algorithms or construction of sensing dictionaries. In principle one can perform similar analysis for general measures satisfying \eqref{eq:submatrix}, this problem is however beyond the scope of this article.

Our Theorem \ref{thm:matrix-sampling} generalizes the estimate by Ruetz and Schnass to arbitrary (not necessarily homogeneous) vectors with the SRP. Moreover, even for the conditioned Bernoulli case, replacing the matrix $\Pi$ by $P$ may be in some cases beneficial. In general, $p_i \le 2\pi_i$  for all $i$ (see \cite[Lemma 3.5]{MR4296760}) and it is not difficult to construct examples in which most of the  entries of $P$ are much smaller than the corresponding entries of $\Pi$. One could also argue that $P$, defined in terms of inclusion probabilities of atoms is a more natural parameter than $\Pi$.

On the other hand, the result in \cite{MR4296760} gives explicit constants, while ours is expressed in terms of absolute constants. We remark that by tracking down the constants in the proofs of all our inequalities, one could obtain explicit expressions. However, since the constants obtained this way would still be rather large, tracking their value would obscure the main arguments, and we do not pursue specific applications of Theorem \ref{thm:matrix-sampling}, we prefer not to specify them.

\section{Proofs}\label{sec:proofs}

\subsection{Preliminary lemmas}
Let us start with two standard linear-algebraic lemmas. Elementary proofs of both of them can be found in \cite{kaufman2022scalar}.
\begin{lemma}\label{le:op-conv} For any $A,B\in \MS$,
\begin{displaymath}
  -A^2 - B^2 \preceq AB+BA \preceq A^2 + B^2
\end{displaymath}
and
\begin{displaymath}
  (A+B)^2 \preceq 2A^2 + 2B^2.
\end{displaymath}
\end{lemma}

\begin{lemma}\label{le:matrix-exp}
If $A,B \in \MS$, $A \preceq \id$ and $0 \preceq B$, then
\begin{displaymath}
  \exp(A-B) \preceq \id + A - B + 2A^   2+2B^2.
\end{displaymath}
\end{lemma}

Let us also recall the classical Golden--Thompson inequality \cite{MR189691,MR189688}.

\begin{theorem}
For any $A,B \in \MS$,
\begin{displaymath}
  \tr e^{A+B} \le \tr e^{A}e^{B}.
\end{displaymath}
\end{theorem}

The next lemma is a part of a more general observation made in \cite[Lemma 4.5]{kaufman2022scalar}. After rearranging the terms and an application of the Bayes rule, it is a quick consequence of the definition of $\ell_\infty$-independence.  For notational simplicity in the lemma below and subsequent part of the article we will often assign arbitrary value to $\p(A|B)$, when $\p(B) = 0$, if such expressions in our formulas are multiplied by $\p(B)$.

\begin{lemma}\label{le:av-max-independence} Let $X$ be a $\{0,1\}^n$-valued, $k$-homogeneous random vector. If $X$ is $\ell_\infty$-independent with parameter $D$, then for any $i\in [n]$,
\begin{displaymath}
  \sum_{v=1}^n \p(X_v=1)|\p(X_i=1|X_v=1) - \p(X_i=1)| \le D\p(X_i=1).
\end{displaymath}
\end{lemma}

We will also need the following lemma which has already been announced in Remark \ref{re:D-at-least-1}

\begin{lemma}\label{le:D-lower-bound}
Let $X$ be a $k$-homogeneous $\{0,1\}^n$-valued random vector, whose law is not a Dirac mass. If $X$ is one-sided $\ell_\infty$-independent with constant $D$, then $D \ge 1$.
\end{lemma}

\begin{proof}
  Let $\Lambda \subset [n]$ be a maximal set such that the measure $\p(X \in \cdot|\forall_{\ell \in \Lambda} X_\ell = 1)$ is well defined and is not a Dirac mass. Note that it may happen that $\Lambda$ is empty, moreover by homogeneity $s := |\Lambda| < k$. Let $I_1,I_2,\ldots,I_m\subseteq [n]\setminus \Lambda$ be \emph{all} pairwise distinct sets of cardinality $k-s > 0$, such that $\p(\forall_{\ell \in I_r} X_\ell = 1|\forall_{\ell\in \Lambda} X_\ell = 1) > 0$.  By the definition of $\Lambda$, $m \ge 2$. Moreover, the sets $I_r$, $r \in [m]$ are pairwise distinct since if there exists $j \in I_{r_1}\cap I_{r_2}$ for $r_1\neq r_2$, then $\p(X \in \cdot|\forall_{\ell \in \Lambda\cup\{j\}} X_\ell = 1)$ is not a Dirac mass, contradicting the maximality of $\Lambda$. By homogeneity, the events $\{\forall_{\ell \in I_r\cup \Lambda} X_{\ell} = 1\}$, $r \in [m]$, are disjoint and $\sum_{r=1}^m\p(\forall_{\ell \in I_r} X_\ell = 1|\forall_{\ell \in \Lambda} X_\ell = 1) = 1$. Denoting $p_r := \p(\forall_{\ell\in I_r} X_\ell = 1|\forall_{\ell\in \Lambda} X_\ell = 1)$, we thus get $p_r \le 1/2$ for some $r\in [m]$. We may without loss of generality assume that $p_1 \le 1/2$.

  Fix $j_1\in I_1,\ldots,j_m\in I_m$ and let $i = j_1$. We have $\p(X_{j_r} = 1| X_i = 1, \forall_{\ell \in \Lambda} X_\ell = 1) = \1_{\{r = 1\}}$. Moreover $\p(X_{j_r} = 1|\forall_{\ell \in \Lambda}X_\ell = 1) = p_r$. Thus
  \begin{multline*}
    \sum_{j=1}^n |\mathcal{I}_{\mu_X}^{\Lambda}(i,j)|
     \ge \sum_{r=1}^m |\p(X_{j_r} = 1|X_i = 1, \forall_{\ell \in \Lambda} X_\ell = 1) - \p(X_{j_r} = 1|\forall_{\ell \in \Lambda} X_\ell = 1)|\\
    = |1 - p_1| + |0 - p_2|+\ldots + |0 - p_m| = 2(1-p_1) \ge 1,
  \end{multline*}
where in the third equality  we used that $\sum_{r=1}^m p_r = 1$ and in the last inequality that $p_1 \le 1/2$.
This shows that $\|\mathcal{I}_{\mu_X}^{\Lambda}\|_{\ell_\infty \to \ell_\infty} \ge 1$ and ends the proof of the lemma.
\end{proof}

\subsection{Proof of results from Section \ref{sec:Bernstein}}\label{sec:proofs-Bernstein}

The proof will rely on the following lemma
\begin{lemma}\label{le:Bernstein-main-lemma} In the setting of Theorem \ref{thm:Bernstein}, assume that $R = 1$. There exists an absolute constant $C$ such that for $K=CD^2$ and every $\lambda \in [-\frac{1}{4K},\frac{1}{4K}]$,
\begin{displaymath}
  \E \tr \exp(\lambda (Z - \E Z)) \le \tr \exp\Big(\lambda^2 K\E \sum_{i=1}^n X_i A_i^2\Big).
\end{displaymath}
\end{lemma}

In our proofs we will use a sampling construction from \cite{kaufman2022scalar}. We remark that it is equivalent to sampling non-zero coordinates of $X$ in a uniformly random order, and in this version in the case of SCP measures it was used already by Pemantle and Peres  \cite{MR3197973}, and later, e.g., in \cite{MR3899605,MR4683375}. In \cite{kaufman2022scalar} it is presented in a slightly different language as a two stage sampling procedure for the law of $X$, allowing to decompose $\mu_X$ into a mixture of conditional $(k-1)$-homogeneous distributions and giving rise to an induction approach. We prefer to present this idea in terms of conditional expectations, starting with the random vector $X$ as this form will be more convenient for us when dealing with SCP measures and more general functions than linear combinations of matrices.

Let $V$ be a random element of $[n]$, which conditionally on $X$ is sampled uniformly from the set $\supp(X) = \{i\in [n]\colon X_i = 1\}$. By $k$-homogeneity it follows that for every $v \in [n]$,
\begin{align}\label{eq:law-of-V}
  \p(V = v) = \frac{\p(X_v=1)}{k}.
\end{align}

It is also easy to see that
\begin{align}\label{eq:conditioning-on-V}
\p(X \in \cdot| V = v) = \p(X \in \cdot|X_v = 1).
\end{align}
In particular, it follows that if $X$ is $\ell_\infty$-independent with parameter $D$, then conditionally on $V$, the random vector $X_{\{V\}^c}$ with values in $\{0,1\}^{\{V\}^c}$ is $(k-1)$-homogeneous and also $\ell_\infty$-independent with parameter $D$.

In the proof of Lemma \ref{le:Bernstein-main-lemma} we will need the following two random variables
\begin{align}\label{eq:definition-of-Z_V}
Z_V &= \E \Big(\sum_{i=1}^n X_i A_i|V\Big) - \E \sum_{i=1}^n X_i A_i,\\
\widehat{Z}_V & = \E \Big(\sum_{i=1}^n X_i A_i^2|V\Big) - \E \sum_{i=1}^n X_i A_i^2.
\end{align}

\begin{lemma}\label{le:Z_V-estimates}
Under the assumptions of Lemma \ref{le:Bernstein-main-lemma},
\begin{align}\label{eq:Z_v-norm-bound}
& \| Z_V \|, \|\widehat{Z}_V\|  \le D,\\
\label{eq:Z_v-square-bound}
& \E Z_V^2 \preceq D^2 \E A_V^2\textrm{ and } \E \widehat{Z}_V^2 \preceq D^2 \E A_V^4.
\end{align}
\end{lemma}

\begin{proof}

We will first prove \eqref{eq:Z_v-norm-bound}. It is enough to consider the inequality for $Z_V$. The estimate on $\widehat{Z}_V$ follows from it by substituting $A_i^2$ in place of $A_i$ (note that $\|A_i^2\| \le 1$).

Using the assumption that $\|A_i\|\le 1$ and \eqref{eq:conditioning-on-V}, we obtain that on the event of the form $\{V = v\}$ and positive probability,
\begin{multline*}
\| Z_V \| = \Big\|\sum_{i=1}^n A_i \bigl(\p(X_i = 1|X_v = 1) - \p(X_i = 1)\bigr)\Big\|\\
\le \sum_{i=1}^n \|A_i\| |\p(X_i = 1|X_v = 1) - \p(X_i = 1))| \\
\le \sum_{i=1}^n |\p(X_i = 1|X_v = 1) - \p(X_i = 1)| \le D
\end{multline*}
by $\ell_\infty$-independence.

We will now proceed with the proof of \eqref{eq:Z_v-square-bound}. Again, it is enough to prove the first inequality, the second one follows by substituting $A_i^2$ for $A_i$.

On the event $\{V = v\}$, we have denoting $p(i) = \p(X_i = 1)$, $p_v(i) = \p(X_i = 1|X_v = 1)$,
\begin{multline*}
  Z_V^2 = \Big(\sum_{i=1}^n A_i (p_v(i) - p(i)\Big)^2 \\
  = \sum_{i=1}^n (p_v(i) - p(i))^2 A_i^2 + \sum_{1\le i < j \le n} (p_v(i) - p(i))(p_v(j) - p(j))(A_i A_j + A_j A_i)\\
  \preceq \sum_{i=1}^n (p_v(i) - p(i))^2 A_i^2  + \sum_{1\le i < j \le n} |p_v(i) - p(i)||p_v(i) - p(i)| (A_i^2 + A_j^2)\\
  = \sum_{i=1}^n |p_v(i) - p(i)| A_i^2 \sum_{j=1}^n |p_v(i) - p(i)| \preceq D \sum_{i=1}^n |p_v(i) - p(i)|A_i^2,
\end{multline*}
where in the first inequality we used Lemma \ref{le:op-conv} and in the last one the $\ell_\infty$-independence.
Thus, by  \eqref{eq:law-of-V} and Lemma \ref{le:av-max-independence},
\begin{multline*}
  \E Z_V^2 \preceq D \sum_{v=1}^n \p(V=v) \sum_{i=1}^n |p_v(i) - p(i)| A_i^2\\
  = D \sum_{v=1}^n \frac{p(v)}{k} \sum_{i=1}^n |p_v(i) - p(i)| A_i^2 \\
  = D \sum_{i=1}^n A_i^2 \sum_{v=1}^n \frac{p(v)}{k}|p_v(i) - p(i)|\\
  \preceq D \sum_{i=1}^n A_i^2 D \frac{p(i)}{k} = D^2 \sum_{i=1}^n A_i^2 \p(V = i) = D^2 \E A_V^2,
\end{multline*}
which proves \eqref{eq:Z_v-square-bound}.

\end{proof}

\begin{lemma}\label{le:norm-bound}
Let $M_1,M_2,M_3$ be $\MS$-valued random variables, such that $\E M_1 = \E M_2 = 0$, $0\preceq M_3 \preceq \id$ , and for some $D \ge 1$,
\begin{align}
& \| M_1 \|, \|M_2\|  \le D,\label{eq:norm-bound}\\
\label{eq:square-bound}
& \E M_1^2 \preceq D^2 \E M_3\textrm{ and } \E M_2^2 \preceq D^2 \E M_3.
\end{align}
There exists a universal constant $C$, such that for any $K \ge CD^2$ and any $\lambda \in (-\frac{1}{4K},\frac{1}{4K})$,
\begin{align}\label{eq:Z_v-estimates}
\Big\| \E \exp\Big(\lambda M_1 + K\lambda^2 M_2 - K\lambda^2 M_3\Big)\Big\| \le 1.
\end{align}
\end{lemma}

\begin{proof}
By \eqref{eq:norm-bound} and the assumption $K |\lambda| \le 1/4$, we get
\begin{align*}
\| \lambda M_1 + K\lambda^2 M_2\| \le (|\lambda| + K\lambda^2)D  \le \frac{5}{16} \frac{D}{K} = \frac{5}{16CD} \le 1
\end{align*}
if $C \ge 5/16$.

Thus, using Lemma \ref{le:matrix-exp} with $A = \lambda M_1 + K\lambda^2 M_2$, $B = K \lambda^2 M_3$, we obtain
\begin{align*}
\E &\exp\Big(\lambda M_1 + K\lambda^2 M_2 - K\lambda^2 M_3\Big)\\
&\preceq \id + \lambda \E M_1 + K\lambda^2 \E M_2 - K\lambda^2 \E M_3 + 2 \E (\lambda M_1 + K\lambda^2 M_2)^2 + 2 K^2\lambda^4\E M_3^2\\
& = \id - K\lambda^2 \E M_3 + 2 \E (\lambda M_1 + K\lambda^2 M_2)^2 + 2 K^2\lambda^4 \E M_3^2\\
&\preceq \id - K\lambda^2 \E M_3 + 4 \lambda^2 \E M_1^2 + 4 K^2\lambda^4 \E M_2^2 + 2K^2\lambda^4 \E M_3^2\\
&\preceq \id - K\lambda^2 \E M_3 + 4 \lambda^2 D^2 \E M_3 + 4D^2 K^2 \lambda^4 \E M_3 + 2K^2\lambda^4 \E M_3^2
\end{align*}
where in the equality we used the assumption $\E M_1 = \E M_2 =0$, in the second inequality Lemma \ref{le:op-conv} and in the last one the assumption \eqref{eq:square-bound}. Taking into account that $0 \preceq M_3 \preceq \id$, we have
\begin{multline*}
- K\lambda^2 \E M_3 + 4 \lambda^2 D^2 \E M_3 + 4D^2 K^2 \lambda^4 \E M_3 + 2K^2\lambda^4 \E M_3^2 \\
\preceq (-K\lambda^2 + 4\lambda^2 D^2 + 4D^2K^2 \lambda^4 + 2K^2\lambda^4)\E M_3.
\end{multline*}
We have $|\lambda| K \le 1/4$, so
\begin{displaymath}
  -K\lambda^2 + 4\lambda^2 D^2 + 4D^2K^2 \lambda^4 + 2K^2\lambda^4 \le -K\lambda^2 + 4\lambda^2 D^2 + \frac{1}{4} D^2  \lambda^2 + \frac{1}{8}\lambda^2  \le 0
\end{displaymath}
for $K \ge 35D^2/8$ (recall that $D \ge 1$).

We thus get $\E \exp\Big(\lambda M_1 + K\lambda^2 M_2 - K\lambda^2 M_3\Big) \preceq \id$, which ends the proof since the left-hand side is nonnegative definite.
We remark that the value of $C$ obtained from the proof is $35/8$.
\end{proof}

\begin{cor}\label{cor:exponential-estimate}
Under the assumptions of Lemma \ref{le:Bernstein-main-lemma} there exists an absolute constant $C$ such that for any $K\ge CD^2$ and every $\lambda \in [-\frac{1}{4K},\frac{1}{4K}$,
\begin{align}\label{eq:exponential-estimate}
\Big\| \E \exp\Big(\lambda Z_V + K\lambda^2 \widehat{Z}_V - K\lambda^2 A_V^2\Big)\Big\| \le 1.
\end{align}
\end{cor}

\begin{proof} We may assume that $X$ is not deterministic, since otherwise $Z_V = \widehat{Z}_V = 0$ and the lemma is trivial. Thus $D \ge 1$ (see Lemma \ref{le:D-lower-bound}).

Now it is enough to apply Lemma \ref{le:norm-bound} with $M_1 = Z_V$, $M_2 = \widehat{Z}_V$ and $M_3 = A_V^2$. The assumptions of Lemma \ref{le:norm-bound} follow from Lemma \ref{le:Z_V-estimates} and the observation that $\|A_i\|\le 1$ for all $i$ implies that $0\preceq A_V^4 \preceq A_V^2 = M_3 \preceq \id$.
\end{proof}

\begin{proof}[Proof of Lemma \ref{le:Bernstein-main-lemma}]
We will prove by induction on $k$ that for every deterministic matrix $H$,
\begin{align}\label{eq:trace-induction}
  \E \tr \exp(H + \lambda Z) \le \tr \exp\Big(H + \lambda \E Z + K \lambda^2\E \sum_{i=1}^n X_i A_i^2\Big).
\end{align}
For $k = 0$ the statement is trivial as both sides of \eqref{eq:trace-induction} equal $\tr e^H$. Let us thus assume that $k \ge 1$ and the statement holds for $(k-1)$.
Let $V$ be the random element of $[n]$ defined above. Conditionally on $V$, $X_{\{V\}^c}$ is $\ell_\infty$-independent with parameter $D$ and $(k-1)$-homogeneous. Thus, using the induction assumption with the matrix $H' = H + \lambda A_V$, measurable with respect to $\sigma(V)$, we may write
\begin{align*}
  \E \tr &\exp(H + \lambda Z)  = \E \E \Big(\tr \exp\Big(H + \lambda  A_V + \sum_{i\neq V} X_i A_i\Big)|V\Big)  \\
  & \le \E \tr \exp\Big(H + \lambda  A_V + \lambda \E (\sum_{i\neq V} X_i A_i|V) + K\lambda^2 \E (\sum_{i\neq V} X_i A_i^2|V)\Big)\\
  & = \E \tr \exp\Big(H + \lambda \E (\sum_{i=1}^n X_i A_i|V) + K\lambda^2 \E (\sum_{i= 1}^n X_i A_i^2|V) - K\lambda^2 A_V^2\Big)\\
  & = \E \tr \exp\Big(H + \lambda \E Z + K \lambda^2\E \sum_{i=1}^n X_i A_i^2 + \lambda Z_V + K\lambda^2 \widehat{Z}_V  - K\lambda^2 A_V^2\Big),
\end{align*}
where in the second equality we used the fact that by construction $X_V = 1$.

Using the Golden-Thompson inequality, the estimate $\tr AB \le \tr A \|B\|$ valid for nonnegative definite $A$, and Lemma \ref{le:norm-bound}, we further estimate

\begin{multline*}
\E \tr \exp(H + \lambda Z)  \\
\le \E \tr \exp\Big(H + \lambda \E Z + K \lambda^2\E \sum_{i=1}^n X_i A_i^2\Big) \exp\Big(\lambda Z_V + K\lambda^2 \widehat{Z}_V - K\lambda^2A_V\Big)\\
= \tr \exp\Big(H + \lambda \E Z + K \lambda^2\E \sum_{i=1}^n X_i A_i^2\Big) \E \exp\Big(\lambda Z_V + K\lambda^2 \widehat{Z}_V - K\lambda^2A_V\Big)\\
\le \tr \exp\Big(H + \lambda \E Z + K \lambda^2\E \sum_{i=1}^n X_i A_i^2\Big)  \Big\| \E \exp\Big(\lambda Z_V + K\lambda^2 \widehat{Z}_V - K\lambda^2 A_V^2\Big)\Big\|\\
\le \tr \exp\Big(H + \lambda \E Z + K \lambda^2\E \sum_{i=1}^n X_i A_i^2\Big),
\end{multline*}
which ends the proof of the induction step. The proof of the lemma is completed by substituting $H = -\lambda \E Z$.
\end{proof}

\begin{proof}[Proof of Theorem \ref{thm:Bernstein}] Having Lemma \ref{le:Bernstein-main-lemma}, we can proceed in the usual way, using exponential Chebyshev's inequality.
By homogeneity we may assume that $R = 1$.

We have for $\lambda > 0$,
\begin{multline*}
  \p(\|Z - \E Z\| \ge t) \\
  \le \p\Bigl(\tr \exp(\lambda (Z - \E Z)) \ge \exp(\lambda t)\Bigr) + \p\Bigl( \tr \exp(-\lambda (Z - \E Z)) \ge \exp(\lambda t)\Bigr)\\
  \le e^{-\lambda t}\E \tr \exp(\lambda (Z - \E Z)) + e^{-\lambda t} \E \tr \exp(-\lambda (Z - \E Z)).
\end{multline*}
Thus, by Lemma \ref{le:av-max-independence}, if $0 < \lambda \le 1/(4K)$, then
\begin{displaymath}
  \p(\|Z - \E Z\| \ge t) \le 2e^{-\lambda t}\tr \exp(\lambda^2 K\E \sum_{i=1}^n X_i A_i^2) \le 2d \exp(K\sigma^2 \lambda^2 - \lambda t).
\end{displaymath}
If $t \le \sigma^2/2$, we substitute  $\lambda = \frac{t}{2K\sigma^2} \le \frac{1}{4K}$ and obtain
\begin{displaymath}
  \p(\|Z - \E Z\| \ge t) \le 2d \exp(-t^2/(4K\sigma^2)).
\end{displaymath}
If $t > \sigma^2/2$, then we substitute $\lambda = 1/(4K)$ and get
\begin{displaymath}
  \p(\|Z - \E Z\| \ge t) \le 2d \exp(\sigma^2/(16 K) - t/(4K)) \le 2d \exp(-t/(8K)).
\end{displaymath}
This ends the proof.
\end{proof}

\begin{proof}[Proof of Theorem \ref{cor:Bernstein-non-hom}]
Consider the homogenization of $X$, i.e., a $\{0,1\}^{2n}$-valued random vector $Y = (X_1,\ldots,X_n,1-X_1,\ldots,1-X_n)$. It is clearly $n$-homogeneous.
Moreover, in the proof of \cite[Corollary 1.4.]{kaufman2022scalar}, Kaufman, Kyng and Solda demonstrate that if $X$ is two-sided $\ell_\infty$-independent with constant $D$, then $Y$ is one-sided $\ell_\infty$-independent with constant $2D$. The corollary now follows immediately by Theorem \ref{thm:Bernstein} applied to $Y$ instead of $X$ and matrices $\widetilde{A}_i = A_i$ for $i\le n$, $\widetilde{A}_i = 0$ for $i = n+1,\ldots,2n$.
\end{proof}

\subsection{Proofs of results from Section \ref{sec:SCP}}

The proof will rely on a combination of the ideas leading to Theorem \ref{thm:Bernstein} with additional arguments, similar to those in the proof of \cite[Theorem 2.8]{MR4683375}, exploiting the SCP. Recall that by \cite[Proposition 3.2]{kaufman2022scalar}, every homogeneous SCP measure is two-sided $\ell_\infty$-independent with $D = 2$.

We will prove the following fact which is a counterpart of Lemma \ref{le:Bernstein-main-lemma}.
\begin{lemma}\label{le:SCP-main-lemma}
In the setting of Theorem \ref{thm:general-function-SCP}, assume that $R=1$. Then there exists an absolute constant $C$, such that for any $\lambda \in [\frac{-1}{4C},\frac{1}{4C}]$,

\begin{displaymath}
\E \tr \exp\bigl(\lambda(f(X) - \E f(X))\bigr) \le \tr \exp\bigl(\lambda^2 C \E \sum_{i=1}^n X_i A_i^2 \bigr).
\end{displaymath}
\end{lemma}

Similarly as in Section \ref{sec:proofs-Bernstein} let us introduce the random variables
\begin{align}\label{eq:definition-of-Z_V-SCP}
Z_V &= \E (f(X)|V) - \E f(X),\\
\widehat{Z}_V & = \E \Big(\sum_{i=1}^n X_i A_i^2|V\Big) - \E \sum_{i=1}^n X_i A_i^2.
\end{align}
where $V$ again is an $[n]$-valued random variable, which conditionally on $X$ is uniformly distributed on $\{i\in [n]\colon X_i = 1\}$.

We then have the following lemma.
\begin{lemma}\label{le:Z_V-estimates-SCP}
Under the assumptions of Lemma \ref{le:SCP-main-lemma}, there exists a universal constant $C$, such that for all $\lambda \in [-\frac{1}{4C},\frac{1}{4C}]$,
\begin{align}\label{eq:Z_v-estimates-SCP}
\Big\| \E \exp\Big(\lambda Z_V + C\lambda^2 \widehat{Z}_V - C\lambda^2 A_V^2\Big)\Big\| \le 1.
\end{align}
\end{lemma}

\begin{proof}
Observe that by Lemma \ref{le:Z_V-estimates}, we have
\begin{align}\label{eq:Z_v-square-bound-SCP}
\|\widehat{Z}_V\|\le 2\textrm{ and } \E \widehat{Z}_V^2 \preceq 4 \E A_V^4\preceq 4\E A_V^2.
\end{align}

We will prove that also
\begin{align}\label{eq:Z_V-bound-SCP}
  \|Z_V\| \le 2, \textrm{ and } \E Z_V^2 \preceq 4 \E A_V^2.
\end{align}

By the SCP assumption, there exists a coupling $(Q^v,Y^v)$ of $\p(X \in \cdot |X_{v}=0)$ and $\p(X \in \cdot |X_{v}=1)$ such that $(Q^v_i)_{i\in [v]^c} \triangleright (Y^v_i)_{i\in [v]^c}$. It follows from $k$-homogeneity that there exists a random variable $U_v$ with values in $[v]^c$, such that $Q^v = Y^v - e_v +e_{U_v}$.

We have
\begin{displaymath}
\E f(X) = \E (f(X) | X_{v}=1)(1-\p(X_{v}=0)) + \E(f(X) | X_{v}=0)\p(X_{v}=0).
\end{displaymath}
Moreover, by \eqref{eq:conditioning-on-V}, on the event $\{V = v\}$ we have $\p(X \in \cdot |V) = \p(X \in \cdot |X_v = 1)$ and so
$\E(f(X)|V=v) = \E(f(X)|X_v=1)$. Combining this with the inequality above, we get

\begin{multline}\label{eq:conditonal-expectation}
    \E (f(X) | V=v) - \E f(X) = \p(X_{v}=0) \bigl( \E (f(X) | X_{v}=1) - \E (f(X) | X_{v}=0)\bigr) \\
    = \p(X_{v}=0) \bigl( \E (f(Y^v) - f(Q^v)) \bigr).
\end{multline}

Denote $a_v = \E A_{U_v}^2$. As a consequence of the above equality, using twice the operator convexity of $x\mapsto x^2$ (see, e.g., \cite[Example V.1.3]{MR1477662}) and the assumption \eqref{eq:increment-assumption} , we get that on $\{V = v\}$,

\begin{multline}\label{eq:pointwise-bound-on-Z_V}
  Z_V^2 = (\E (f(X) | V=v) - \E f(X))^2 \\
  = \p(X_v = 0)^2  \Big( \bigl(\E (f(Y^v) -  f(Y^v-e_v))\bigr) + \bigl(\E (f(Y^v - e_v) - f(Y^v - e_v + e_{U_v}))\bigr) \Big)^2 \\
  \preceq 2\p(X_v = 0)^2  \Big[\Big( \E (f(Y^v) - f(Y^v-e_v))\Big)^2 + \Big(\E (f(Y^v - e_v) + f(Y^v - e_v + e_{U_v}))\Big)^2 \Big]\\
  \preceq 2\p(X_v = 0)^2 \Big[ \E \bigl(f(Y^v) - f(Y^v-e_v)\bigr)^2 + \E \bigl(f(Y^v - e_v) + f(Y^v - e_v + e_{U_v})\bigl)^2\Big]\\
  \preceq 2\p(X_v = 0)^2 (A_v^2 + \E A_{U_v}^2) = 2\p(X_v = 0)^2 (A_v^2 + a_v)\\
  \preceq 2\p(X_v = 0)(A_v^2 + a_v).
\end{multline}

Applying \eqref{eq:conditonal-expectation} to the function $g\colon \{0,1\}^n\to \MS$ given by $g(x) = \sum_{i=1}^n x_i A_i^2$ instead of $f$, we get
\begin{multline*}
\E \Big(\sum_{i=1}^nX_iA_i^{2}|V=v\Big) - \E \Big(\sum_{i=1}^nX_iA_i^{2}\Big)= \p(X_{v}=0)\E \sum_{i=1}^n A_i^2 (Y^v_i - Q^v_i) \\
= \p(X_{v}=0)\E (A_v^2 - A_{U_v}^2\bigr) = \p(X_{v}=0)\bigl(A_v^2 - a_v\bigr).
\end{multline*}

Integrating this equality with respect to the distribution of $V$ we obtain
\begin{displaymath}
  \sum_{v=1}^n \p(V=v)\p(X_v=0)A_v^2 =   \sum_{v=1}^n \p(V=v)\p(X_v=0)a_v.
\end{displaymath}

Going back to \eqref{eq:pointwise-bound-on-Z_V}, we get,
\begin{multline*}
  \E Z_V^2 \preceq  \sum_{v=1}^n 2 \p(V = v) \p(X_v = 0)(A_v^2 + a_v)\\
   = 4 \E \sum_{v=1}^n \p(V=v)\p(X_v=0)A_v^2 \preceq 4 \E A_V^2.
\end{multline*}

This implies the second estimate of \eqref{eq:Z_V-bound-SCP}. The first one follows from \eqref{eq:pointwise-bound-on-Z_V} and the observation that $\|A_v\|, \|a_v\| \le 1$.

Having \eqref{eq:pointwise-bound-on-Z_V} and \eqref{eq:Z_V-bound-SCP}, to obtain \eqref{eq:Z_v-estimates-SCP} it is enough to observe that $\E Z_V = \E \widehat{Z}_V = 0$ and apply Lemma \ref{le:norm-bound} with $D= 2$ and $M_1 = Z_V$, $M_2 = \widehat{Z}_V$, $M_3 = A_V^2$.
\end{proof}

We are now ready for

\begin{proof}[Proof of Lemma \ref{le:SCP-main-lemma}]
The argument is very similar to that in the proof of Lemma \ref{le:Bernstein-main-lemma}. It requires only notational changes and an application of Lemma \ref{le:Z_V-estimates-SCP}. We present the details for completeness.

We will show by induction on $k$ that for every deterministic matrix $H$,
\begin{align}\label{eq:trace-induction-SCP}
  \E \tr \exp(H + \lambda f(Z)) \le \tr \exp\Big(H + \lambda \E f(Z) + C \lambda^2\E \sum_{i=1}^n X_i A_i^2\Big).
\end{align}
For $k = 0$ the statement is trivial as both sides of \eqref{eq:trace-induction-SCP} equal $\tr e^{H + f(0)}$. Let us thus assume that $k \ge 1$ and the statement holds for $(k-1)$.
Let $V$ be the random element of $[n]$ defined above. Conditionally on $\{V=v\}$, $X_{\{v\}^c}$ is a $(k-1)$-homogeneous measure with the SCP and so it is $\ell_\infty$-independent with parameter $D=2$ . Thus, using the induction assumption, we get
\begin{align*}
  \E \tr &\exp(H + \lambda f(X))  = \E \E \Big(\tr \exp(H + \lambda f(X)|V\Big)  \\
  & \le \E \tr \exp\Big(H + \lambda \E (f(X)|V) + C\lambda^2 \E (\sum_{i\neq V} X_i A_i^2|V)\Big)\\
  & = \E \tr \exp\Big(H + \lambda \E f(X) + \lambda Z_V + C\lambda^2 \E (\sum_{i= 1}^n X_i A_i^2|V) - C\lambda^2 A_V^2\Big)\\
  & = \E \tr \exp\Big(H + \lambda \E f(X) + C \lambda^2\E \sum_{i=1}^n X_i A_i^2 + \lambda Z_V + C\lambda^2 \widehat{Z}_V  - C\lambda^2 A_V^2\Big),
\end{align*}
where in the second equality we used the fact that by construction $X_V = 1$.

Using the Golden-Thompson inequality and inequality \eqref{eq:Z_v-estimates-SCP} of Lemma \ref{le:Z_V-estimates-SCP}, we further estimate
\begin{multline*}
\E \tr \exp(H + \lambda f(X))  \\
\le \E \tr \exp\Big(H + \lambda \E f(X) + C \lambda^2\E \sum_{i=1}^n X_i A_i^2\Big) \exp\Big(\lambda Z_V + C\lambda^2 \widehat{Z}_V - C\lambda^2A_V^2\Big)\\
=\tr \exp\Big(H + \lambda \E f(X) + C \lambda^2\E \sum_{i=1}^n X_i A_i^2\Big) \E \exp\Big(\lambda Z_V + C\lambda^2 \widehat{Z}_V-C\lambda^2A_V^2\Big)\\
\le \tr \exp\Big(H + \lambda \E f(X) + C \lambda^2\E \sum_{i=1}^n X_i A_i^2\Big)\Big\| \E \exp\Big(\lambda Z_V + C\lambda^2 \widehat{Z}_V - C\lambda^2 A_V^2\Big)\Big\|\\
\le \tr \exp\Big(H + \lambda \E f(X) + C \lambda^2\E \sum_{i=1}^n X_i A_i^2\Big),
\end{multline*}
which ends the proof of the induction step. The proof of the lemma is completed by substituting $H = -\lambda \E Z$.
\end{proof}

\begin{proof}[Proof of Theorem \ref{thm:general-function-SCP}]The theorem follows from Lemma \ref{le:SCP-main-lemma} in exactly the same way as Theorem \ref{thm:Bernstein} from Lemma \ref{le:Bernstein-main-lemma} (one formally replaces $Z$ with $f(X)$ and $K$ with $C$).
\end{proof}

\subsection{Proofs of results from Section \ref{sec:decoupling}}

Let us start with the proof of the decoupling inequality.

\begin{proof}[Proof of Theorem \ref{thm:decoupling}] Consider binary random variables $\delta_i$, $i\in [n]$, such that
\begin{displaymath}
    \p(\delta_i = 1|X) = \frac{1}{2}X_i
\end{displaymath}
and $\delta_i$ are conditionally independent given $X$.

Using the fact that $c_{ii} = 0$, together with conditional independence and the equality $\E(\delta_i|X) = \frac{1}{2}X_i$, we have
\begin{multline*}
\E\Big(\sum_{i,j=1}^n c_{ij}X_i X_j \delta_i (1-\delta_j)\Big|X\Big) =   \sum_{i,j=1}^n c_{ij} X_i X_j \E(\delta_i|X)\E(1-\delta_j|X) \\
= \sum_{i,j=1}^n \frac{1}{2}c_{ij}X_i^2X_j(1-\frac{1}{2}X_j) = \frac{1}{4}\sum_{i,j=1}^n c_{ij} X_i X_j,
\end{multline*}
where in the last equality we used the fact that $X_i$'s take only values $0$ and $1$.
Conditionally on $X$, the random variable
\begin{displaymath}
    Z = \Big\|\sum_{i,j=1}^n c_{ij}X_i X_j \delta_i (1-\delta_j) \Big\|
\end{displaymath}
is an $E$-valued tetrahedral polynomial of degree 2 in independent Rademacher variables $\varepsilon_i = 2\delta_i - 1$, $i \in \supp(X)$. Therefore, by \cite[Theorem 3.2.5]{MR1666908},
there exists a universal constant $K$, such that
\begin{displaymath}
    (\E(Z^2|X))^{1/2} \le K\E (Z|X).
\end{displaymath}
It follows by the Paley--Zygmund inequality (see, e.g., \cite[Corollary 3.3.2]{MR1666908}) that
\begin{displaymath}
    \p\Big(Z \ge 2^{-1}\E(Z|X)\Big|X\Big) \ge \frac{1}{4} \frac{(\E(Z|X))^2}{\E(Z^2|X)} \ge \frac{1}{4K^2}.
\end{displaymath}
Since by Jensen's inequality
\begin{displaymath}
    \E(Z|X) \ge  \Big \|\E(\sum_{i,j=1}^n c_{ij}X_i X_j \delta_i (1-\delta_j)|X)\Big\| = \frac{1}{4}\Big\|\sum_{i,j=1}^n c_{ij} X_i X_j\Big\|,
\end{displaymath}
we thus get
\begin{multline}\label{eq:decoupling-first-estimate}
    \p(Z \ge t) = \E \p(Z \ge t|X) \ge \E \p(Z \ge t|X)\1_{\{\|\sum_{i,j=1}^n c_{ij} X_i X_j\| \ge 8t\}} \\
    \ge \E \p\Big(Z \ge 2^{-1}\E(Z|X)\Big|X\Big)\1_{\{\|\sum_{i,j=1}^n c_{ij} X_i X_j\| \ge 8t\}} \\
    \ge \frac{1}{4K^2} \p\Big(\Big\|\sum_{i,j=1}^n c_{ij} X_i X_j\Big\| \ge 8t\Big).
\end{multline}

Let $\sigma(X)$ be the $\sigma$-field generated by $X$. Consider the random set $I = \{i\in [n]\colon \delta_i = 1\}$. Note that for any non-empty $J \subseteq [n]$, such that $\p(I = J) > 0$, and any event $A \in \sigma(X)$,
\begin{align}\label{eq:conditioning}
    \p(A|I = J) = \p(A|\forall_{i\in J} X_i = 1).
\end{align}
Indeed, it is enough to consider $A$ being atoms of $\sigma(X)$, i.e., sets of the form
\begin{displaymath}
A = \{X_1 = x_1,\ldots,X_n=x_n\}
\end{displaymath}
for some $x_1,\ldots,x_n \in \{0,1\}$, exactly $k$ of which are equal to 1. Moreover, since $\p(\delta_i = 1|X_i = 0) = 0$, both sides of \eqref{eq:conditioning} vanish if for some $j \in J$, $x_j = 0$. Thus, to prove \eqref{eq:conditioning}, it is enough to show that for every $x_j \in \{0,1\}$, $j \in J^c$, with $k - |J|$ ones
\begin{align*}
\p(\forall_{j \in J^c}  X_j = x_j\textrm{ and } \forall_{j \in J} X_j = 1 |I = J) = \p(\forall_{j \in J^c} X_j = x_j | \forall_{j \in J} X_j = 1 ).
\end{align*}

Using the definition of conditional probability and rearranging, we obtain that the above equality is equivalent to
\begin{multline*}
    \p(\forall_{j \in J^c} X_j = x_j\textrm{ and } \forall_{j \in J} X_j = 1 \textrm{ and } I = J)\p(\forall_{j \in J} X_j = 1 )\\
    =  \p(\forall_{j \in J^c} X_j = x_j\textrm{  and } \forall_{j \in J} X_j = 1 )\p(I = J).
\end{multline*}
Using homogeneity and the definition of $\delta_i$, we obtain that the left-hand side above equals
\begin{displaymath}
\p(\forall_{j \in J^c} X_j = x_j\textrm{  and } \forall_{j \in J} X_j = 1 )2^{-k} \p(\forall_{j \in J} X_j = 1 )
\end{displaymath}
and so does the right hand side, since similarly
\begin{displaymath}
\p(I = J) = \sum_{K\supseteq J, |K| = k} \p(\forall_{j \in K} X_j = 1)2^{-k} = \p(\forall_{j \in J} X_j = 1 )2^{-k}.
\end{displaymath}
This establishes \eqref{eq:conditioning}.

Now, on the event $I = J$ we have
\begin{displaymath}
    Z = \Big\|\sum_{i\in J,j\in J^c} c_{ij}X_i X_j\Big\| = \Big\|\sum_{i\in J,j\in J^c} c_{ij} X_j\Big\|
\end{displaymath}
 and so we can write for $t > 0$,
\begin{multline}\label{eq:preparation-for-NA}
    \p(Z \ge t) = \sum_{\stackrel{\emptyset \neq J \subseteq [n]}{|J| \le k}}\p(Z \ge t|I = J) \p(I = J)\\
= \sum_{\stackrel{\emptyset \neq J \subseteq [n]}{|J| \le  k}}\p\Big( \Big\|\sum_{i\in J,j\in J^c} c_{ij}X_j\Big\|\ge t\Big|I = J\Big) \p(I = J)\\
= \sum_{\stackrel{\emptyset \neq J \subseteq [n]}{|J| \le  k}}\p\Big( \Big\|\sum_{i\in J,j\in J^c} c_{ij} X_j\Big\|\ge t\Big |\forall_{j \in J} X_j = 1 \Big)\p(I = J),
\end{multline}
where in the last equality we used \eqref{eq:conditioning} and the fact that the event appearing there belongs to $\sigma(X)$.

Observe that the random variables
$\1_{\{\|\sum_{i\in J,j\in J^c} c_{ij}X_j\| \ge t\}}$ and $\1_{\{\forall_{j \in J} X_j = 1 \}}$ are non-decreasing functions of $X_{J^c}$ and $X_J$ respectively (we use here the assumed monotonicity of the bilinear form $f$). Thus, by negative association,
\begin{displaymath}
  \E \1_{\{\|\sum_{i\in J,j\in J^c} c_{ij}X_j\| \ge t\}}\1_{\{\forall_{j \in J} X_j = 1 \}} \le \E \1_{\{\|\sum_{i\in J,j\in J^c} c_{ij}X_j\| \ge t\}} \E \1_{\{\forall_{j \in J} X_j = 1 \}}
\end{displaymath}

Let now $Y$ be a copy of $X$, independent of $X$ and $\delta_i$'s. For any $J$, such that $\p(X_j = 1 \textrm{ for all $j \in J $}) > 0$ (equivalently $\p(I = J) > 0$), we have
\begin{multline*}
    \p\Big(\Big\|\sum_{i\in J,j\in J^c} c_{ij}X_j\Big\| \ge t\Big|\forall_{j \in J} X_j = 1 \Big) \\
    = \frac{\E \1_{\{\|\sum_{i\in J,j\in J^c} c_{ij}X_j\| \ge t\}}\1_{\{\forall_{j \in J} X_j = 1 \}}}{\E \1_{\{\forall_{j \in J} X_j = 1 \}}} \\
    \le \E \1_{\{\|\sum_{i\in J,j\in J^c} c_{ij}X_j\| \ge t\}} = \p\Big(\Big\|\sum_{i\in J,j\in J^c} c_{ij}X_j\Big\| \ge t\Big)\\
    = \p\Big(\Big\|\sum_{i\in J,j\in J^c} c_{ij}Y_j\Big\| \ge t\Big) \le
    \p\Big(\Big\|\sum_{i\in J,1\le j\le n} c_{ij}Y_j\Big\| \ge t\Big),
\end{multline*}
where in the last inequality we again used the monotonicity of $f$.
Thus, by \eqref{eq:preparation-for-NA},
\begin{multline*}
    \p(Z \ge t) \le \sum_{\stackrel{\emptyset \neq J \subseteq [n]}{|J| \le k}} \p\Big( \Big\|\sum_{i\in J,1\le j\le n} c_{ij} Y_j\Big\|\ge t\Big)\p(I = J)\\
    = \sum_{\stackrel{\emptyset \neq J \subseteq [n]}{|J| \le k}} \p\Big( \Big\|\sum_{i\in J,1\le j\le n} c_{ij} Y_j\Big\|\ge t\textrm{ and } I = J\Big)\\
    = \sum_{\stackrel{\emptyset \neq J \subseteq [n]}{|J| \le k}} \p\Big( \Big\|\sum_{i\in J,1\le j\le n} c_{ij} X_i Y_j\Big\|\ge t\textrm{ and } I = J\Big)\\ \le \sum_{\stackrel{\emptyset \neq J \subseteq [n]}{|J| \le k}} \p\Big( \Big\|\sum_{i,j=1}^n c_{ij} X_i Y_j\Big\|\ge t\textrm{ and } I = J\Big)
    \le \p\Big( \Big\|\sum_{i,j=1}^n c_{ij} X_i Y_j\Big\|\ge t\Big),
\end{multline*}
where in the first equality we used  the independence of $Y$ and $\delta$, in the second one the fact that if $I = J$, then $X_i = 1$ for $i \in J$, finally in the second inequality we used once more the monotonicity of $f$.

Going back to \eqref{eq:decoupling-first-estimate}, we obtain
\begin{displaymath}
  \p\Big(\Big\|\sum_{i,j=1}^n c_{ij} X_i X_j\Big\| \ge 8t\Big) \le 4K^2 \p\Big( \Big\|\sum_{i,j=1}^n c_{ij} X_i Y_j\Big\|\ge t\Big),
\end{displaymath}
which ends the proof.
\end{proof}

We will now prove Theorem \ref{thm:matrix-sampling}
\begin{proof}[Proof of Theorem \ref{thm:matrix-sampling}]
The norm $\|H_X\|$ can be written as
\begin{displaymath}
  \|H_X\| = \Big\|\sum_{1\le i\neq j\le d} e_ie_j^T h_{ij} X_i X_j\Big\|
\end{displaymath}
and it is easy to see that the function
\begin{displaymath}
\{0,1\}^d\times \{0,1\}^d\ni (x,y) \mapsto \Big\|\sum_{1\le i\neq j\le d} e_ie_j^T h_{ij} x_i y_j\Big\|
\end{displaymath}
is coordinate-wise non-decreasing. Therefore, by Corollary \ref{cor:decoupling-Rayleigh}, for $Y$ -- an independent copy of $X$, we have
\begin{align}\label{eq:decoupling-submatrix}
  \p(\|H_X\| \ge t) \le C\p(\|\Lambda_X H\Lambda_Y\| \ge t/C),
\end{align}
where $\Lambda_Y = Diag(Y_1,\ldots,Y_d)$.

We may and will assume that $X,Y$ are defined as coordinates on a product probability space. To simplify the notation for conditioning, we will from now on denote by $\p_X,\p_Y$ the probability with respect to the $X$, resp. $Y$ coordinate on this space, with the other coordinate fixed, i.e., the conditional probability given $Y$, resp. $X$. An analogous standard convention will be used for $\E_X, \E_Y$.

Let us first bound the conditional probability $\p_Y(\|\Lambda_X H \Lambda_Y\|\ge t)$. Denote the columns of $\Lambda_X H$ by $W_i$. Note that they are random vectors measurable with respect to $X$ and thus independent of $Y$. Using the fact that $\Lambda_Y \Lambda_Y^T = \Lambda_Y^2 = \Lambda_Y$, we get
\begin{displaymath}
  Z := \Lambda_X H \Lambda_Y (\Lambda_X H \Lambda_Y)^T = \sum_{i=1}^d Y_i W_i W_i^T.
\end{displaymath}
Recall that for any square matrix $A$, $\|A\|^2 = \|AA^T\| = \|A^TA\|$.
Thus, applying Corollary \ref{cor:KKS-bound} to $Y$, conditionally on $X$, we obtain
\begin{multline}\label{eq:P_Y-bound}
  \p_Y(\|\Lambda_X H \Lambda_Y\| \ge \sqrt{\|\E_Y Z\|  + t^2}) = \p_Y( \|Z\| \ge \|\E_Y Z\| + t^2) \\
  \le   2d\exp\Big(-\frac{1}{C}\min\Big(\frac{t^4}{\|\E_Y Z\| \max_{i\le d} |W_i|^2},\frac{t^2}{\max_{i\le d} |W_i|^2}\Big)\Big).
\end{multline}
Now, we observe that
\begin{displaymath}
  \E_Y Z = \sum_{i=1}^d p_i W_i W_i^T = \Lambda_X H \sqrt{P}  (\Lambda_X H\sqrt{P})^T.
\end{displaymath}
Therefore,
\begin{displaymath}
  \|\E_Y Z\| =   \|(\Lambda_X H\sqrt{P})^T \Lambda_X H \sqrt{P}\| = \Big\|\sum_{i=1}^d X_i U_iU_i^T\Big\|,
\end{displaymath}
where $U_i$ is the $i$-the column of $\sqrt{P} H^T$. Moreover, we have
\begin{displaymath}
  \Big\| \E \sum_{i=1}^d X_i U_iU_i^T \Big\|= \| (\sqrt{P} H^T \sqrt{P})(\sqrt{P} H\sqrt{P})\| = \|\sqrt{P}H\sqrt{P}\|^2
\end{displaymath}
Applying again Corollary \ref{cor:KKS-bound}, we get
\begin{multline}\label{eq:bound-on-E_Y}
  \p\Big(\|\E_Y Z\| \ge \|\sqrt{P}H\sqrt{P}\|^2 + t^2\Big) \\
  \le 2d\exp\Big(-\frac{1}{C}\min\Big(\frac{t^4}{\|\sqrt{P}H\sqrt{P}\|^2 \max_{i\le d} |U_i|^2},\frac{t^2}{\max_{i\le d} |U_i|^2}\Big)\Big)\\
  = 2d\exp\Big(-\frac{1}{C}\min\Big(\frac{t^4}{\|\sqrt{P}H\sqrt{P}\|^2 \|H\sqrt{P}\|_{\ell_2\to \ell_\infty}^2},\frac{t^2}{\|H\sqrt{P}\|_{\ell_2\to \ell_\infty}^2}\Big)\Big).
\end{multline}

The last ingredient we need is a deviation bound on $\max_{i\le n}|W_i|$. For a fixed index $i$, we get
\begin{displaymath}
  |W_i|^2 = \sum_{j=1}^d X_i h_{ij}^2
\end{displaymath}
and, denoting by $H_i$ the $i$-th column of $H$,
\begin{displaymath}
\E \sum_{j=1}^d X_i h_{ij}^2 = |\sqrt{P} H_i|^2,
\end{displaymath}
which by an application of Corollary \ref{cor:KKS-bound} in dimension 1, gives
\begin{displaymath}
  \p(|W_i|^2 \ge |\sqrt{P} H_i|^2 + t\max_{j\le d}|h_{ij}|)\le 2\exp\Big(-\frac{1}{C}\min\Big(\frac{t^2}{|\sqrt{P} H_i|^2},\frac{t}{\max_{j\le d} |h_{ij}|}\Big)\Big).
\end{displaymath}
Taking the union bound over all $i \le d$, we obtain
\begin{multline}\label{eq:bound-on_W}
\p(\max_{i\le d} |W_i|^2 \ge \|\sqrt{P} H\|_{\ell_1 \to \ell_2}^2 + t\|H\|_{\ell_1\to \ell_\infty}) \\
\le 2d\exp\Big(-\frac{1}{C}\min\Big(\frac{t^2}{\|\sqrt{P} H\|_{\ell_1\to \ell_2}^2 },\frac{t}{\|H\|_{\ell_1\to \ell_\infty}}\Big)\Big).
\end{multline}

Assume now that $t \ge \|\sqrt{P}H\sqrt{P}\|$. Let
\begin{displaymath}
\mathcal{A} = \{\|\E_Y Z\| < 2t^2, \max_{i\le d} |W_i|^2 \le \|\sqrt{P} H\|_{\ell_1\to \ell_2}^2 + t\|H\|_{\ell_1\to \ell_\infty}\}.
\end{displaymath}
Using the union bound together with the Fubini theorem, we can write
\begin{align*}
  \p\Bigl(&\|\Lambda_X H \Lambda_Y\| \ge \sqrt{3}t\Bigr) \\
   \le& \p\Bigl(\Bigl\{\|\Lambda_X H \Lambda_Y\| \ge \sqrt{3t^2}\Bigr\}\cap \mathcal{A}\Bigl)
  +\p\Bigl(\|\E_Y Z\| > 2t^2\Bigr) \\
  & \phantom{aaaaaaaaa}+ \p\Bigl(\max_{i\le d} |W_i|^2 \ge \|\sqrt{P} H\|_{\ell_1 \to \ell_2}^2 + t\|H\|_{\ell_1\to \ell_\infty}\Bigr)\\
  \le&
  \E_X \1_{\mathcal{A}} \p_Y\Bigl(\|\Lambda_X H \Lambda_Y\| \ge \sqrt{\|\E_Y Z\| + t^2}\Bigr) \\
  &\phantom{aaaaaaaaa}+ \p\Bigl(\|\E_Y Z\| > \|\sqrt{P}H\sqrt{P}\| + t^2 \Bigr) \\
  &\phantom{aaaaaaaaaaaaaa}+ \p\Bigl(\max_{i\le d} |W_i|^2 \ge \|\sqrt{P} H\|_{\ell_1 \to \ell_2}^2 + t\|H\|_{\ell_1\to \ell_\infty}\Bigr).
\end{align*}

Using \eqref{eq:P_Y-bound}, \eqref{eq:bound-on-E_Y} and \eqref{eq:bound-on_W} to bound from above the last three probabilities, we obtain
\begin{align*}
\p\Bigl(&\|\Lambda_X H \Lambda_Y\| \ge \sqrt{3}t\Bigr) \\
\le& 2d \E_X \1_{\mathcal{A}}\exp\Big(-\frac{1}{C}\min\Big(\frac{t^4}{\|\E_Y Z\| \max_{i\le d} |W_i|^2},\frac{t^2}{\max_{i\le d} |W_i|^2}\Big)\Big)\\
&+   2d\exp\Big(-\frac{1}{C}\min\Big(\frac{t^4}{\|\sqrt{P}H\sqrt{P}\|^2 \|H\sqrt{P}\|_{\ell_2\to \ell_\infty}^2},\frac{t^2}{\|H\sqrt{P}\|_{\ell_2\to \ell_\infty}^2}\Big)\Big)\\
& + 2d\exp\Big(-\frac{1}{C}\min\Big(\frac{t^2}{\|\sqrt{P} H\|_{\ell_1\to \ell_2}^2 },\frac{t}{\|H\|_{\ell_1\to \ell_\infty}}\Big)\Big),
\end{align*}
which by the definition of $\mathcal{A}$ and the inequality $t \ge \|\sqrt{P}H\sqrt{P}\|$ gives
\begin{multline*}
  \p\Bigl(\|\Lambda_X H \Lambda_Y\| \ge \sqrt{3}t\Bigr) \\
  \le 6d\exp\Big(-\frac{1}{C'}\min\Big(\frac{t^2}{\|\sqrt{P}H\|_{\ell_1\to \ell_2}^2},\frac{t^2}{\|H\sqrt{P}\|_{\ell_2\to \ell_\infty}^2},\frac{t}{\|H\|_{\ell_1\to\ell_\infty}}\Big)\Big).
\end{multline*}
The theorem follows now by \eqref{eq:decoupling-submatrix} and an adjustment of constants.
\end{proof}

\bibliographystyle{amsplain}	
\bibliography{matrix-discrete-cube.bib}

\end{document}